\documentclass{amsart}
\usepackage{bullclass,graphicx}
\usepackage{mathrsfs}

\usepackage{amsthm,cite}

\newcommand{\UCV}{\mathcal{UCV}}

\newcommand{\al}{\alpha}
\newcommand{\s}{\mathcal{SL}}
\newcommand{\p}{\mathcal{S}_{\mathcal{P}}}
\newcommand{\U}{\mathcal{UCV}}
\newcommand{\m}{\mathcal{M}(\beta)}
\newcommand{\C}{\mathcal{C}(\alpha)}
\newcommand{\T}{\mathcal{S}^*(\alpha)}
\DeclareMathOperator{\RE}{Re}

\begin{document}
\title[On the Radius Constants for Classes of Analytic Function] {On the Radius Constants for Classes of Analytic Functions}

\author[R. M. Ali, N. K. Jain \and V. Ravichandran]{\first Rosihan M. Ali, \second Naveen Kumar Jain \and \third V. Ravichandran}

\address{\first \andthird  School of Mathematical Sciences, Universiti Sains Malaysia, 11800 USM, Penang, Malaysia\\
 \second \andthird Department of Mathematics, University of Delhi, Delhi 110007, India\\}
\email{\first rosihan@cs.usm.my, \second naveenjain05@gmail.com, \third vravi@maths.du.ac.in}

\thanks{The work presented here was supported in parts by a Research University grant from Universiti Sains Malaysia, and by a research fellowship from
the Council of Scientific and Industrial Research, New Delhi.}

\begin{abstract}

Radius constants for several classes of analytic functions on the unit disk are obtained. These include the radius of starlikeness of a positive order, radius of parabolic starlikeness, radius of Bernoulli lemniscate starlikeness, and radius of uniform convexity. In the main, the radius constants obtained are sharp. Conjectures on the non-sharp constants are given.
\end{abstract}


\recdate{June 1,  2012??}{August  16, 2012??}{}{Anton Abdulbasah Kamil}

\maketitle

\section{Introduction} This paper studies the class $\mathcal{A}$ of analytic functions $f$ in $\mathbb{D} =\{z\in\mathbb{C}: |z| <1\}$
normalized by $f(0)=0=f'(0)-1$. Let $\mathcal{S}$ be its subclass consisting of univalent functions. For $0\leq\alpha<1$, let $\mathcal{S^*}(\alpha)$  and $\mathcal{C}(\alpha) $ be the subclasses of $\mathcal{S}$ consisting respectively of functions starlike of order $\alpha$ and convex   of order $\alpha$. These are functions respectively characterized by $\text{\textrm{Re}}\ zf'(z)/ f(z)> \alpha$ and $1 + \text{\textrm{Re}} \
 zf''(z)/ f'(z)> \alpha$. The usual classes of starlike and convex functions are denoted by $\mathcal{S^*}:=\mathcal{S^*}(0)$  and $\mathcal{C}:=\mathcal{C}(0)$.

The Koebe function $k(z)=z/(1-z)^2$, which maps $\mathbb{D}$ onto the region $\mathbb{C}\setminus \{ w: w\leq -1/4\}$, is starlike but not convex. However, it is known that $k$ maps the disk $\mathbb{D}_r:=\{z\in\mathbb{D} : |z|<r \}$ onto a convex domain for every $r\leq 2-\sqrt{3}$. Indeed, every univalent function $f\in \mathcal{S}$ maps $\mathbb{D}_r$ onto a convex region for $r\leq 2-\sqrt{3}$ \cite[Theorem 2.13, p.\ 44]{duren}. This number is called the radius of convexity for $\mathcal{S}$.

In general, for two families $\mathcal{G}$ and $\mathcal{F}$ of $\mathcal{A}$, the \emph{$\mathcal{G}$-radius} of $\mathcal{F}$, denoted by $R_\mathcal{G}(\mathcal{F})$, is the largest number $R$ such that $r^{-1}f(rz)\in\mathcal{G}$ for $0< r\leq R$, and $f \in \mathcal{F}$. Whenever $\mathcal{G}$ is characterized by possesing a geometric property $P$, the number $R$ is also referred to as the radius of property $P$ for the class $\mathcal{F}$.

Several other subclasses of $\mathcal{A}$ and $\mathcal{S}$ are also of great interest. In \cite{kaplan}, Kaplan introduced the close-to-convex functions $f\in \mathcal{A}$ satisfying $f'(z)\neq0$ and $\RE (f'(z)/g'(z))>0$ for some (not necessarily normalized) convex function $g$. In his investigation on the Bieberbach conjecture for close-to-convex functions, Reade \cite{reade} introduced the class of close-to-starlike functions. These are functions $f\in \mathcal{A}$ with $f(z)\neq0$ in $\mathbb{D}\setminus\{0\}$ and $\RE (f(z)/g(z))>0$ for a (not necessarily normalized) starlike function $g$. Close-to-convex functions are known to be univalent, but close-to-starlike functions need not.  There are various other studies on classes of functions in $\mathcal{A}$ characterized by the ratio between functions $f$ and $g$ belonging to particular subclasses of $\mathcal{A}$ \cite{bajpai,bhargava,causey,chen,goel,goelb,goel72,krzyz,ksp1,ksp2,mac2,mac1,mac4,mac3,ratti68,ratti70,%
ratti80,reade,reade65,singh,vravicmft,shah,shukla,tuan}.

Radius constants have been obtained for several of these subclasses. In \cite{mac4, mac3}, MacGregor obtained the radius of starlikeness for the class of functions $f\in\mathcal{A}$ satisfying either
\begin{equation}\label{ee3}
\RE\left(\frac{f(z)}{g(z)}\right)>0
\quad \text{or}\quad
 \left|\frac{f(z)}{g(z)}-1\right|<1
\end{equation} in $ \mathbb{D}$
for some $g\in\mathcal{C}$.  Ratti \cite{ratti68} determined the radius of starlikeness for functions $f$ belonging to a variant class of \eqref{ee3}.
In \cite{mac2},   MacGregor  found the radius of convexity for univalent functions satisfying $|f'(z)-1|<1$, while Ratti \cite{ratti70} established the radius for functions $f$  satisfying
\begin{equation*}
 \left|\frac{f'(z)}{g'(z)}-1\right|<1\quad ( z\in \mathbb{D})
\end{equation*}
when $g$ belongs to certain classes of analytic functions.

This paper finds radius constants for several classes of functions $f\in \mathcal{A}$ characterized by its ratio with a certain function $g$. In the following section, the classes consisting of uniformly convex functions, parabolic starlike functions, and Bernoulli lemniscate starlike functions will be brought fore to attention. In the main, the real part of the involved expressions lie in the right half-plane, and so in Section 1.2, we shall gather certain results involving functions of positive real part that will be required. Section 2 contains the main results involving the radius of Bernoulli lemniscate starlikeness, radius of starlikeness of positive order, and radius of parabolic starlikeness for several classes. These include the subclasses satisfying one of the conditions: (i)  $\RE(f(z)/g(z))>0$ where $\RE (g(z)/z)>0$ or $\RE(g(z)/z)>1/2$, (ii) $ \left|(f(z)/g(z))-1\right|$ $<1$ where $\RE(g(z)/z)>0$ or $g$ is convex, and (iii)  $ \left|(f'(z)/g'(z))-1\right|<1$  where $\RE g'(z)>0$. Section 3 is devoted to finding the radius of uniform convexity for the classes $ \left|(f'(z)/g'(z))-1\right|<1$, and $g$ is either univalent, starlike or convex.

\subsection{Subclasses of univalent functions}\label{sec11}
This section highlights certain important subclasses of $f\in\mathcal{S}$ that will be referred to in the sequel.
A function $f\in\mathcal{S}$ is \emph{uniformly convex}  if for every
circular arc $\gamma$ contained in $\mathbb{D}$ with center
$\zeta\in \mathbb{D}$, the image arc $f(\gamma)$ is convex. The class $\mathcal{UCV}$
of all uniformly convex  functions  was introduced by Goodman  \cite{ucv}.
In two separate papers, R\o nning \cite{ronn} and Ma and Minda \cite{maminda-ucv}
independently proved that
\begin{equation*}\label{e13}
f\in \mathcal{UCV} \Longleftrightarrow \RE \left( 1 +
\frac{zf''(z)}{f'(z)} \right) > \left|\frac{zf''(z)}{f'(z)} \right|
\quad ( z\in \mathbb{D}).
\end{equation*}
R\o nning introduced a corresponding class of starlike functions
called parabolic starlike functions. These are functions $f\in\mathcal{A}$ satisfying
\begin{equation*}\label{eq17} \RE\left(\frac{zf'(z)}{f(z)} \right)
>\left|\frac{zf'(z)}{f(z)}-1 \right|\quad ( z\in \mathbb{D}).
  \end{equation*}
Denote the class of such functions by $\mathcal{S}_\mathcal{ P}$. It is  evident that $f\in \UCV$ if and only if $zf'(z)\in \p$. A recent survey on these classes can be found in \cite{raviucv}.  The class $ \mathcal{S}_\mathcal{ L}$, introduced by
Sok\'o\l\ and Stankiewicz \cite{sokol96}, consists of functions $f\in\mathcal{A}$ satisfying the inequality
\[
 \left|\left(\frac{zf'(z)}{f(z)}\right)^2-1\right|<1 \quad ( z\in \mathbb{D}).
\]
Thus a function $f$ is in the class $\mathcal{S}_\mathcal{ L}$  if $zf'(z)/f(z)$ lies in the
region bounded by the right-half of the lemniscate of Bernoulli
given by $|w^2-1|<1$. Results related to the  class  $\mathcal{S}_\mathcal{ L}$ can be found in \cite{nav,ali1,sokol09,sokol09b}.  Another class  $ \mathcal{M}(\beta)$, $\beta > 1$, consisting of functions $f\in\mathcal{A}$ satisfying
\[  \RE\left( \frac{zf'(z)}{f(z)}  \right) < \beta\quad ( z\in \mathbb{D}), \]   was investigated by Uralegaddi
{\em et al.} \cite{ural}  and  Owa and Srivastava \cite{owa2}.

\subsection{On functions with positive real part}\label{sec12}
For $ 0\leq \alpha<1$, let $\mathcal{P}(\alpha)$  denote the class of functions $p(z)=1+c_1z+\cdots $ satisfying the inequality $\RE(p(z))>\alpha$ in $\mathbb{D}$ and write $ \mathcal{P}:=\mathcal{P}(0)$. This class is related to various subclasses in  $\mathcal{S}$ as will be evident in the next section. The following results for functions in $\mathcal{P}(\alpha)$ will be required in the sequel.

\begin{lemma}\label{lm5}{\rm\cite{vraviron}}
Let $p\in\mathcal{P}(\alpha)$, and $|z|=r<1$. Then
\[\left|p(z)-\frac{1+(1-2\al)r^2}{1-r^2}\right|\leq\frac{2(1-\al)r}{1-r^2}  .\]
\end{lemma}

\begin{lemma}\label{lm3}{\rm\cite{shah}}
Let $p\in\mathcal{P}(\alpha)$, and $|z|=r<1$. Then
\[\left|\frac{zp'(z)}{p(z)}\right|\leq\frac{2r(1-\alpha)}{(1-r)[1+(1-2\alpha)r]} .\]
\end{lemma}


\begin{lemma}\label{lm7}{\rm \cite[Lemma 2.4]{causey}}
Let $p\in\mathcal{P}(1/2)$, and $|z|=r<1$. Then
\[\RE \frac{zp'(z)}{p(z)} \geq\begin{cases}
 - r/(1+r), &  r< 1/3,\\
 - (\sqrt{2}-\sqrt{1-r^2})^2/(1-r^2), &  1/3 \leq r\leq \sqrt{8\sqrt{2}-11}\approx0.56.\end{cases}
 \]
\end{lemma}


\begin{lemma}\label{lm4}{\rm\cite{nav}}  Let $0 < a< \sqrt{2}$. If  $r_a$ is given by
\[ r_a= \begin{cases} \left(\sqrt{1-a^2}-(1-a^2)\right)^{1/2}, &\quad  0 <a \leq 2 \sqrt{2}/3 \\
\sqrt{2}-a, & \quad 2 \sqrt{2}/3 \leq a <\sqrt{2},
\end{cases} \]  
then  \[ \{ w: |w-a|< r_a\} \subseteq \{w: |w^2-1|<1\}.\] 
\end{lemma}

\begin{lemma}\label{lm6}{\rm\cite{vravicmft}} Let $a> 1/2$. If the number $R_a$ satisfies
\[R_a=\begin{cases}
a- 1/2, &  1/2 < a  \leq  3/2  \\
\sqrt{2a-2}, & a \geq  3/2 ,
\end{cases} \]
then  \[ \{ w: |w-a|< R_a\} \subseteq \{w: |w-1|<\RE w\}. \]
\end{lemma}

\section{Radius Constants for Analytic Functions}
Let $\mathcal{F}_1$ be the class of functions $f\in \mathcal{A}$ satisfying the inequality
\[ \RE\left(\frac{f(z)}{g(z)}\right)>0\quad (z\in \mathbb{D})\]
for some $g\in \mathcal{A}$ with \[ \RE \left(\frac{g(z)}{z}\right)>0\quad (z\in \mathbb{D}).\] Ratti \cite{ratti68} showed that the radius of starlikeness of functions in $\mathcal{F}_1$  is $\sqrt{5}-2\approx 0.2360$ and that the radius can be improved to 1/3 if the function $g$ additionally satisfies $\RE(g(z)/z)>1/2$.

\begin{theorem}\label{th7} For the class $\mathcal{F}_1$,   the following sharp radius results hold:
 \begin{enumerate}
 \item[(a)]\label{i3} the $ \s $-radius for $\mathcal{F}_1$ is \[R_\s= \frac{\sqrt2-1}{2+\sqrt{7-2\sqrt2}}\simeq0.10247,\]
 \item[(b)]\label{i4}the $ \m $-radius for $\mathcal{F}_1$ is \[R_{\m}= \frac{\beta-1}{2+\sqrt{4+(\beta-1)^2}},\]
 \item[(c)]\label{i27}  the $\T $-radius for $\mathcal{F}_1$ is \[R_{\T}= \frac{ 1-\alpha }{2+\sqrt{5+\alpha ^2-2\alpha }},\]
 \item[(d)]\label{i5}   the $\p $-radius for $\mathcal{F}_1$ is \[R_{\p}=R_{\mathcal{S}^*{( 1/2 )}}=\frac{ 1}{4+\sqrt{17}}\simeq0.12311.\]
 \end{enumerate}
\end{theorem}
\begin{proof} (a) Let $f\in \mathcal{F}_1$ and
define $p,h:\mathbb{D}\rightarrow \mathbb{C}$ by
\[p(z)=\frac{g(z)}{z}\quad \text{ and } \quad h(z)=\frac{f(z)}{g(z)}.\]
 Then $p$, $h \in\mathcal{P} $ and
 $f(z)=g(z)h(z)=zp(z)h(z)$. Thus
 \begin{equation*}\label{e1}
 \frac{zf'(z)}{f(z)}=1+\frac{zp'(z)}{p(z)}+\frac{zh'(z)}{h(z)}.
\end{equation*}
Using  Lemma \ref{lm3}, it follows that
\begin{equation}\label{ee1}
\left|\frac{zf'(z)}{f(z)}-1\right|\leq\frac{4r}{1-r^2}, \quad (|z|=r).
\end{equation}

By Lemma \ref{lm4}, the function $f$ satisfies
 \[ \left|\left(\frac{zf'(z)}{f(z)}\right)^2-1\right|\leq 1  \] provided
\begin{equation*}\label{e2}
\frac{4r}{1-r^2}\leq\sqrt2-1,
\end{equation*}
or
\[(\sqrt2-1)r^2+4r+1-\sqrt2 \leq0.\]
This inequality yields  $r \leq R_\s. $

To show that $R_\s$ is the sharp $\s$-radius for $\mathcal{F}_1$, consider the functions $f_0$ and $g_0$ defined by
\begin{equation}\label{f0g0}
f_0(z)=z\left(\frac{1+z}{1-z}\right)^2 \quad \text{and} \quad g_0(z)=z \left(\frac{1+z}{1-z}\right).
\end{equation}
Since $\RE\left( f_0(z)/g_0(z)\right)=\RE((1+z)/(1-z))>0$  and $  \RE \left(g_0(z)/z\right)>0$,  the function
$f_0 $ belongs to $\mathcal{F}_1$. Now
\[ \frac{zf_0'(z)}{f_0(z)}= 1+\frac{4z}{1-z^2}.\]   For $z=\rho:=R_\s$,
\[ \left|\left(\frac{zf'_0(z)}{f_0(z)}\right)^2-1\right|=\left|\left(1+\frac{4\rho}{1-\rho^2}\right)^2-1\right|=1  .\]
This shows that  the radius in (a)  is sharp.

(b) From inequality \eqref{ee1}, it follows that
\[\RE\frac{zf'(z)}{f(z)}\leq1+\frac{4r}{1-r^2}\leq \beta\]
if\[(1-\beta)+4r-(1-\beta)r^2\leq 0,\]
that is, for $r\leq R_{\m}.$ For the function $f_0$ given by \eqref{f0g0},
\[ \frac{zf_0'(z)}{f_0(z)}= \frac{4\rho+1-\rho^2}{1-\rho^2}=\beta\quad ( z=\rho:=R_{\m} ),\]
 and so the radius is sharp.

(c) Inequality \eqref{ee1} also yields
\[\RE\frac{zf'(z)}{f(z)}\geq1-\frac{4r}{1-r^2}\geq\alpha\]    provided
\[r^2(1-\alpha)+4r-(1-\alpha)\leq0.\]
The last inequality holds whenever  $r\leq R_{\T}$.
The function $f_0$ in \eqref{f0g0} gives
\[ \frac{zf'_0(z)}{f_0(z)}= \frac{ 1+4z-z^2}{1-z^2}=\alpha\]
for $z=-\rho:=-R_{\T}$, and this shows that the radius in (c) is sharp.

(d) In view of Lemma \ref{lm6}, the  circular disk \eqref{ee1} lies completely  inside the parabolic region $\{w:|w-1|<\RE w\}$ provided
\[ \frac{4r}{1-r^2}\leq \frac{1}{2},\]
or
 \begin{equation*}
 r^2+8r-1\leq0.
 \end{equation*}
The last inequality holds whenever  $r\leq R_{\p}=R_{\mathcal{S}^*{( 1/2 )}}= 1/(4+\sqrt{17})$  for $z=-R_{\p}$.

The function $f_0$ in \eqref{f0g0} satisfies
\[\left|\frac{zf_0'(z)}{f_0(z)}-1\right|=\left|\frac{4z}{1-z^2}\right|=
\RE\frac{1-z^2+4z}{1-z^2}=\RE\left(\frac{zf_0'(z)}{f_0(z)}\right),\]
and so the result in (d) is sharp.
\end{proof}

Consider next the class $\mathcal{F}_2$  of functions $f\in \mathcal{A}$ satisfying
\[  \RE\left(\frac{f(z)}{g(z)}\right)>0\quad ( z\in \mathbb{D})\]
for some function $g\in \mathcal{A}$ with
 \[  \RE \left(\frac{g(z)}{z}\right)>\frac{1}{2}\quad ( z\in \mathbb{D}). \]

\begin{theorem}\label{th5}
For the class $\mathcal{F}_2$,    the following radius results hold:
 \begin{enumerate}
  \item[(a)]\label{i1} the $ \s $-radius  is  \[R_\s= \frac{4-2\sqrt2}{ \sqrt2(\sqrt{17-4\sqrt2}+3)}\simeq0.13009,\]

  \item[(b)]\label{i2} the $ \m $-radius  is   \[R_{\m}= \frac{2(\beta-1)}{3+\sqrt{9+4\beta(\beta-1)}},\]

  \item[(c)]\label{ii1}the $ \T $-radius  is \[R_{\T}= \frac{2(1-\alpha )}{3+\sqrt{9-4\alpha +4\alpha ^2}},\]

 \item[(d)]\label{ii5} the $\p $-radius satisfies  \[R_{\p}\geq -3+\sqrt{10}\simeq0.162278.\]
 \end{enumerate}
The radius in (a), (b), and (c) are sharp.
\end{theorem}
\begin{proof}
(a) Let $f\in \mathcal{F}_2$, and
define functions $p,h:\mathbb{D}\rightarrow \mathbb{C}$ by
\[p(z)=\frac{g(z)}{z}\quad \text{ and } \quad h(z)=\frac{f(z)}{g(z)}.\]
Then $f(z)=z h(z)p(z)$ with $h \in \mathcal{P}$ and $p\in \mathcal{P}(1/2).$ Now
\begin{equation}\label{eq10}
\frac{zf'(z)}{f(z)}=1+\frac{zh'(z)}{h(z)}+\frac{zp'(z)}{p(z)}.
\end{equation}
From  Lemma \ref{lm3}, it follows that
\begin{equation}\label{eq10.1}
\left|\frac{zf'(z)}{f(z)}-1\right|\leq\frac{2r}{1-r^2}+\frac{r}{1-r}=\frac{3r+r^2}{1-r^2}.
\end{equation}
 By Lemma \ref{lm4}, the function $f$ satisfies
 \[ \left|\left(\frac{zf'(z)}{f(z)}\right)^2-1\right|\leq 1  \] provided
\[\frac{3r+r^2}{1-r^2}\leq\sqrt2-1,\] or  \[ \sqrt2r^2+3r+1-\sqrt2\leq0.\] This holds whenever   $r \leq R_\s. $

This radius $R_\s$ is the sharp $\s$-radius for $\mathcal{F}_2$. For this purpose, let $f_0$ and $g_0$ be defined by
\begin{equation}\label{2f0g0}
f_0(z)=\frac{z(1+z)}{(1-z)^2}\quad \text{and} \quad g_0(z)=  \frac{z}{1-z} .
\end{equation}
Since $\RE\left( f_0(z)/g_0(z)\right)>0$  and $  \RE \left(g_0(z)/z\right)>1/2$,  the function
$f_0\in \mathcal{F}_2$. Also
\[ \frac{zf_0'(z)}{f_0(z)}= 1+\frac{z(3+z)}{1-z^2}.\]   Thus at $z=\rho:=R_\s$,

\[ \left|\left(\frac{zf'_0(z)}{f_0(z)}\right)^2-1\right|=\left|\left(\frac{1+3z}{1-z^2}\right)^2-1\right|=1.\]

(b) From  inequality \eqref{eq10.1}, it follows that
 \[\RE\frac{zf'(z)}{f(z)}\leq\frac{3r+1}{1-r^2}\leq\beta\]
provided
 \[(1+\beta)r^2+3r+1-\beta\leq0,\] that is, if $r\leq\m.$
For  $f_0$ given by \eqref{2f0g0},
\[ \frac{zf_0'(z)}{f_0(z)}= \frac{1+3\rho}{1-\rho^2}=\beta\quad ( z=\rho:=R_{\m} ),\]
 and so the result in (b) is sharp.


(c) Using Lemmas \ref{lm3},  \ref{lm7} and \eqref{eq10}, it follows that
\begin{equation}\label{eq01}
\RE\left(\frac{zf'(z)}{f(z)}\right)>1-\frac{2r}{1-r^2}-\frac{r}{1+r}=\frac{1-3r}{1-r^2}\geq\alpha
\end{equation}
if
 \[  \alpha-1+3r-\alpha r^2  \leq0.\] The last inequality holds whenever   $r\leq R_{\T} $.
For  $f_0$ given by \eqref{2f0g0},
\[ \frac{zf'_0(z)}{f_0(z)}= \frac{ 1+3z}{1-z^2}=\frac{1-3\rho}{1-\rho^2}=\alpha\]
at $z=-\rho:=-R_{\T}$, and this shows that the result in (c) is sharp.

(d) From \eqref{eq10.1} and \eqref{eq01},  it follows that \[\left|\frac{zf'(z)}{f(z)}-1\right|<\RE\left(\frac{zf'(z)}{f(z)}\right) \] if
\[ \frac{1-3r}{1-r^2}\geq\frac{ 3r+r^2}{1-r^2}, \]
that is \[r^2+6r-1\leq0.\]
 The last inequality holds whenever $r\leq R_{\p}.$
\end{proof}

\begin{conjecture}
The sharp $\p$-radius for $\mathcal{F}_2$ is
\[0.162278\simeq-3+\sqrt{10}\leq R_{\mathcal{S}^*{( 1/2 )}}= R_{\p}=3-2\sqrt2\simeq0.171573.\]
\end{conjecture}

Let $\mathcal{F}_3$ be the class of all functions $f\in \mathcal{A}$ satisfying the inequality
\[  \left|\frac{f(z)}{g(z)}-1 \right|<1 \quad ( z\in \mathbb{D})\]
for some function $g\in \mathcal{A}$ with
\[ \RE \left(\frac{g(z)}{z}\right)>0 \quad ( z\in \mathbb{D}).\]

\begin{theorem}\label{th5b}
For the class  $\mathcal{F}_3$,    the following radius results hold:
 \begin{enumerate}
  \item[(a)]\label{i1} the $ \s $-radius is  \[R_\s= \frac{4-2\sqrt2}{ \sqrt2(\sqrt{17-4\sqrt2}+3)}\simeq0.13009,\]

  \item[(b)]\label{i2} the $ \m $-radius  is   \[R_{\m}= \frac{2(\beta-1)}{3+\sqrt{9+4\beta(\beta-1)}},\]

  \item[(c)]\label{ii2}the $ \T $-radius  is \[R_{\T}=\frac{2(1-\alpha )}{3+\sqrt{9+4(2-\alpha )(1-\alpha )}},\]

  \item[(d)]\label{i6} the $\p $-radius  is  \[R_{\p}=
  R_{\mathcal{S}^*{( 1/2)}}=\frac{2\sqrt{3}-3}{3}\simeq0.154701.\]
 \end{enumerate}
The radii in (c) and (d) are sharp.
\end{theorem}
\begin{proof}
(a) Let $f\in \mathcal{F}_3$. Then
$ \left|f(z)/g(z)-1\right|<1$  if and only if $\RE\{g(z)/f(z)\}>1/2\, (z\in\mathbb{D}).$

Define the functions $p,h:\mathbb{D}\rightarrow \mathbb{C}$ by
\[p(z)=\frac{g(z)}{z}\quad \text{ and } \quad h(z)=\frac{g(z)}{f(z)}.\]
 Then $p\in\mathcal{P} $ and $h \in\mathcal{P}(1/2) $. Now

\begin{equation*}\label{e101}
 \frac{zf'(z)}{f(z)}=1+\frac{zp'(z)}{p(z)}-\frac{zh'(z)}{h(z)},
\end{equation*}
and Lemma \ref{lm3} yields
\begin{equation}\label{e6}
\left|\frac{zf'(z)}{f(z)}-1\right|\leq \frac{r(3+r)}{1-r^2}.
\end{equation}
By Lemma \ref{lm4}, the function $f$ satisfies
 \[ \left|\left(\frac{zf'(z)}{f(z)}\right)^2-1\right|\leq 1  \] provided
\[\frac{3r+r^2}{1-r^2}\leq\sqrt2-1,\] or \[\sqrt2r^2+3r+1-\sqrt2\leq0.\]
Solving this inequality leads to $r \leq R_\s.$

(b) From inequality \eqref{e6}, it follows that
 \[\RE\frac{zf'(z)}{f(z)}\leq\frac{3r+1}{1-r^2}\leq\beta\]
 if
 \[\beta r^2+3r+1-\beta\leq0,\] or  whenever $r\leq R_{\m}.$

(c) Inequality \eqref{e6} also yields

\[\RE\left(\frac{zf'(z)}{f(z)}\right)\geq\frac{1-3r-2r^2}{1-r^2} \geq\alpha\]
if
\[(2-\alpha)r^2+3r+\alpha-1\leq0.\] 
The last inequality holds if  $r\leq R_{\T}$.

To show that $R_{\T}$ is the sharp $\T$-radius for $\mathcal{F}_3$, consider the functions $f_0$ and $g_0$ defined by
\begin{equation}\label{3f0g0}
f_0(z)= \frac{z(1+z)^2}{1-z}  \quad \text{and} \quad g_0(z)=z \left(\frac{1+z}{1-z}\right).
\end{equation}
Since $ \left| f_0(z)/g_0(z)-1\right|=|z|<1$  and $  \RE \left(g_0(z)/z\right)>0$,  the function
$f_0\in \mathcal{F}_3 $. Also
\[ \RE\left(\frac{zf_0'(z)}{f_0(z)}\right)= \RE\left(\frac{1+3z-2z^2}{1-z^2}\right) =\alpha\] at  $z=-\rho:=-R_{\T}$, and
this shows that  the result in (c)  is sharp.

(d) In view of Lemma \ref{lm6}, the  circular disk \eqref{e6} lies completely  inside the parabolic region $\{w:|w-1|<\RE w\}$ if
\[\frac{r(3+r)}{1-r^2}\leq \frac{1}{2},\]
or
\[3r^2+6r-1\leq0.\]

The last inequality  holds if  $r\leq R_{\p}=R_{\mathcal{S}^*{( 1/2 )}}= 1/(4+\sqrt{17})$ for $z=-r:=-R_{\p}$.
The function $f_0$ given by \eqref{3f0g0} satisfies
\[\left|\frac{zf_0'(z)}{f_0(z)}-1\right|=\left|\frac{3r+r^2}{1-r^2} \right|=\frac{1-3r-2r^2}{1-r^2}=\RE\left(\frac{zf'_0(z)}{f_0(z)}\right).\]
Thus the radius in (d) is sharp.
\end{proof}

Let $\mathcal{F}_4$ be the class of functions $f\in \mathcal{A}$ satisfying the inequality
\[   \left|\frac{f'(z)}{g'(z)}-1\right|<1\quad ( z\in \mathbb{D})\]
for some  $g\in \mathcal{A}$ with
   $\RE g'(z)>0\quad ( z\in \mathbb{D}).$ In view of Alexander's relation between $\p$ and $\U$,   the result below  follows from Theorem ~\ref{th5b}.

\begin{theorem}\label{thm4}
For the class $\mathcal{F}_4$,    the following sharp radius results hold:
\begin{enumerate}
\item\label{i25}  the $\C$-radius  is
\[R_{\C}=\frac{2(1-\alpha)}{3+\sqrt{9+4(\alpha-2)(\alpha-1)}},\]
\item\label{i26}   the $\U$-radius  is
\[R_{\U}= R_{\mathcal{C}( 1/2)}= \frac{2\sqrt{3}-3}{3} \simeq0.154701.\]
\end{enumerate}
\end{theorem}

\begin{conjecture}
  The sharp $ \s$-radius and sharp $\m$-radius for the class $\mathcal{F}_3$ are given by
\[R_{\s}=\frac{3}{2}+\frac{3}{2 \sqrt{2}}-\frac{1}{2} \sqrt{\frac{27}{2}+7 \sqrt{2}}\simeq0.142009, \quad
R_{\m}=\frac{2(\beta -1)}{3+\sqrt{9+4(\beta -1)(\beta -2)}}.\]

\end{conjecture}

Let $\mathcal{F}_5$ be the class of all functions $f\in \mathcal{A}$ satisfying the inequality
\[   \left|\frac{f(z)}{g(z)}-1\right|<1\quad ( z\in \mathbb{D})\]
for some   convex function $g\in \mathcal{A}$.

\begin{theorem}\label{th6} 
For the class $\mathcal{F}_5$,    the following radius results hold:
\begin{enumerate}
\item[(a)]\label{i28} the $\T$-radius  is \[R_{\T}=\frac{ 1-\alpha }{1+\sqrt{2+\alpha ^2-2\alpha }},\]
\item[(b)]\label{i7}  the $\p$-radius  is \[R_{\p}=R_{\mathcal{S}^*{( 1/2 )}}=\frac{1}{\sqrt5+2}\simeq0.236068,\]
\item[(c)]\label{i8}  the $\s$-radius   is \[R_{\s}=3-2\sqrt2\simeq0.171573, \]
\item[(d)]\label{i13} the $\m$-radius  is \[R_{\m}= \frac{\beta -1}{1+\beta }.\]
\end{enumerate}
\end{theorem}
\begin{proof}

(a) Let $f\in \mathcal{F}_5$.
 Then $h=g/f \in\mathcal{P}(1/2) $ and
\begin{equation}\label{eq14}
\frac{zf'(z)}{f(z)}=\frac{zg'(z)}{g(z)}-\frac{z h'(z)}{h(z)}.
\end{equation}
Since $g $ is convex, \[\RE\left(\frac{zg'(z)}{g(z)}\right)>\frac{1}{2}.\]
It follows from Lemma \ref{lm5} that
\begin{equation}\label{eq15}
\left|\frac{zg'(z)}{g(z)}-\frac{1}{1-r^2}\right|\leq\frac{r}{1-r^2}.
\end{equation}
Lemma \ref{lm3} together with \eqref{eq14} and \eqref{eq15} gives
\begin{equation}\label{ee5}
\left|\frac{zf'(z)}{f(z)}-\frac{1}{1-r^2}\right|\leq\frac{r}{1-r^2}+\frac{r}{1-r}=\frac{2r+r^2}{1-r^2}.
\end{equation}
Thus
\[\RE\left(\frac{zf'(z)}{f(z)} \right)\geq\frac{1-2r-r^2}{1-r^2}\geq\alpha \] provided
  \[(1-\alpha)r^2+2r+\alpha-1\leq0.\]
 The last inequality  holds if  $r\leq R_{\T}$.

Sharpness of the $\T$-radius for $\mathcal{F}_5$ can be seen by considering the functions $f_0$ and $g_0$ defined by
\begin{equation}\label{5f0g0}
f_0(z)=z\left(\frac{1+z}{1-z}\right) \quad \text{and} \quad g_0(z)=\frac{z}{1-z}.
\end{equation}
Since $\left| f_0(z)/g_0(z)-1\right|=|z|<1$  and $g_0$ is convex,  the function
$f_0\in\mathcal{F}_5 $. Also
\[ \frac{zf_0'(z)}{f_0(z)}= \frac{1-z^2+2z}{1-z^2}=\alpha\] at $z=-r:=-R_{\T}$.

(b) In view of Lemma \ref{lm6}, the  circular disk \eqref{ee5} lies completely  inside the parabolic region $\{w:|w-1|<\RE w\}$ when

 \[\frac{2r+r^2}{1-r^2}\leq\frac{1}{1-r^2}-\frac{1}{2},\]
 or
 \[r^2+4r-1\leq0.\]
The last inequality   holds if  $r\leq R_{\p}=R_{\mathcal{S}^*{( 1/2 )}}$ for $z=-r:=-R_{\p}.$
The function $f_0$ given by \eqref{5f0g0} satisfies
\[\left|\frac{zf_0'(z)}{f_0(z)}-1\right|=\left|\frac{2z}{1-z^2}\right|=\frac{2r}{1-r^2}=\frac{1-r^2-2r}{1-r^2}=\RE\left(\frac{zf_0'(z)}{f_0(z)}\right),\]
and so the radius in (b) is sharp.

(c) By Lemma \ref{lm4} and \eqref{ee5}, the function $f$ satisfies
 \[ \left|\left(\frac{zf'(z)}{f(z)}\right)^2-1\right|\leq 1  \] provided
\[\frac{2r+r^2}{1-r^2}\leq\sqrt2-\frac{1}{1-r^2}, \]
or
\[(\sqrt2+1)r^2+2r-\sqrt2+1\leq0.\]
Solving this inequality yields  $r \leq R_\s.$


(d) From inequality \eqref{ee5}, it follows that
\[\RE\frac{zf'(z)}{f(z)}\leq\frac{2r+r^2+1}{1-r^2}\leq\beta\]
if
\[(1+\beta)r^2+2r+1-\beta\leq0,\]
or if $r\leq R_{\m}.$
\end{proof}

\begin{conjecture}  \label{i11}The sharp $\s$-radius and  $\m$-radius for the class $\mathcal{F}_5$ are given by \[R_{\s}= -1-\sqrt{2}+\sqrt{2 \left(2+\sqrt{2}\right)}\simeq0.198912, \quad R_{\m}=\frac{(\beta -1)}{1+\sqrt{\beta ^2+2-2\beta }}.\]
\end{conjecture}

\section{Radius of Uniform Convexity}
This section considers sharp radius results for classes of functions introduced by   Ratti \cite{ratti80}.
Let $\mathcal{F}_6$ be the class of  functions $f\in \mathcal{A}$ satisfying the inequality
\[   \left|\frac{f'(z)}{g'(z)}-1\right|<1\quad ( z\in \mathbb{D})\]
for some   univalent function $g\in \mathcal{A}$.
\begin{theorem}\label{thm1}
For the class $\mathcal{F}_6$,    the following sharp radius results hold:
\begin{enumerate}
\item[(a)]\label{i19} the  $\C$-radius  is
\[R_{\C}=\frac{2(1-\alpha)}{5+\sqrt{25+4\alpha(\alpha-1)}},\]
\item[(b)]\label{i20} the   $\U$-radius  is
\[R_{\U}= R_{\mathcal{C}( 1/2)}=5-2\sqrt{6}\simeq0.101021 .\]
\end{enumerate}
\end{theorem}

\begin{proof}

(a) Let $f\in \mathcal{F}_6$, and $h:\mathbb{D}\rightarrow \mathbb{C}$ be given by
\[h(z)=\frac{g'(z)}{f'(z)}.\]
 Then  $h \in\mathcal{P}(1/2) $ and
\begin{equation}\label{eq4}
\frac{zf''(z)}{f'(z)}=\frac{zg''(z)}{g'(z)}-\frac{z h'(z)}{h(z)}.
\end{equation}
   Since $g $ is univalent, it is known \cite[Theorem 2.4, p.\ 32]{duren} that
\begin{equation}\label{eq5}
\left|\frac{zg''(z)}{g'(z)}-\frac{2r^2}{1-r^2}\right|\leq\frac{4r}{1-r^2},\quad(|z|=r).
\end{equation}
Now Lemma \ref{lm3}, \eqref{eq4} and \eqref{eq5} yields
\begin{equation}\label{e15}
\left| 1+\frac{zf''(z)}{f'(z)} -\frac{1+r^2}{1-r^2}\right|\leq\frac{5r+r^2}{1-r^2}.
\end{equation}
Thus
\[\RE\left(1+\frac{zf''(z)}{f'(z)}\right) \geq\frac{1-5r}{1-r^2}\geq\alpha\]
if
\[\alpha r^2-5r+1-\alpha\geq0.\]
The last inequality holds when  $r\leq R_{\C}$.

Next consider the functions $f_0$ and $g_0$ defined by
\begin{equation}\label{6f0g0}
f'_0(z)=\frac{(1+z)^2}{(1-z)^3} \quad \text{and} \quad g_0(z)= \frac{z}{(1-z)^2}.
\end{equation}
Since $\left| f'_0(z)/g'_0(z)-1\right|=|z|<1$  and $g_0$ is univalent,  the function
$f_0 \in \mathcal{F}_6$. Also
\[1+\frac{zf''_0(z)}{f'_0(z)}=\frac{1+5z }{1-z^2}.\]
At $z=-r:=-R_{\C}$,
\[\RE\left(1+\frac{zf''_0(z)}{f'_0(z)}\right)=\frac{1-5r }{1-r^2}=\alpha.\]
This shows that  the result in (a)  is sharp.

(b) In view of Lemma \ref{lm6}, the  circular disk \eqref{e15} lies completely  inside the parabolic region $\{w:|w-1|<\RE w\}$ if
\[\frac{5r+r^2}{1-r^2}\leq\frac{1+r^2}{1-r^2}-\frac{1}{2},\]
that is, provided
\begin{equation}\label{eq6.1}
r^2 -10r+1\geq0.
\end{equation}
The last inequality   holds if  $r\leq R_{\U}=R_{\mathcal{C}( 1/2)}=5-2\sqrt{6} $ for $z=-r=-R_{\U}$.
The function $f_0$ given by \eqref{6f0g0} satisfies
\[\left| \frac{zf''_0(z)}{f'_0(z)}\right|  =\frac{r(5-r)}{1-r^2}=\frac{1-5r}{1-r^2}
 =\RE\left(1+\frac{zf''_0(z)}{f'_0(z)}\right),\]
and so the radius (b) is sharp.
\end{proof}

Let $\mathcal{F}_7$ be the class of all functions $f\in \mathcal{A}$ satisfying the inequality
\[   \left|\frac{f'(z)}{g'(z)}-1\right|<1\quad ( z\in \mathbb{D})\]
for some starlike function $g\in \mathcal{A}$.

\begin{theorem}\label{thm2}
For the class $\mathcal{F}_7$,    the following sharp radius results hold:
\begin{enumerate}
\item\label{i21} the $\C$-radius  is
\[R_{\C}=\frac{2(1-\alpha)}{5+\sqrt{25+4\alpha(\alpha-1)}},\]
\item\label{i22} the $\U$-radius  is
\[R_{\U}= R_{\mathcal{C}( 1/2)}=5-2\sqrt{6}\simeq0.101021 .\]
\end{enumerate}
\end{theorem}
\begin{proof}Since $g$ is starlike, it is univalent, and the result follows easily from Theorem ~\ref{thm1}.
\end{proof}

Let $\mathcal{F}_8$ be the class of all functions $f\in \mathcal{A}$ satisfying the inequality
\[   \left|\frac{f'(z)}{g'(z)}-1\right|<1\quad ( z\in \mathbb{D})\]
for some convex function $g\in \mathcal{A}$.

\begin{theorem}\label{thm3}
For the class $\mathcal{F}_8$,    the following radius results hold:
\begin{enumerate}
\item[(a)]\label{i23}the $\C$-radius  is
\[R_{\C}=\frac{2(1-\alpha)}{3+\sqrt{9+4\alpha(\alpha-1)}}.\]
\item[(b)]\label{i24} the $\U$-radius  is
\[R_{\U}= R_{\mathcal{C}( 1/2)}=3-2\sqrt{2}\simeq0.171573 .\]
\end{enumerate}
 The results are sharp.
\end{theorem}
\begin{proof}
(a) The function $g$ is convex, and so is univalent. Proceeding as in the proof of Theorem \ref{thm1}, evidently

\begin{equation}\label{e18}
\left|1+\frac{zf''(z)}{f'(z)} -\frac{1+r^2}{1-r^2}\right|\leq\frac{3r+r^2}{1-r^2},
\end{equation}
which yields
\[\RE\left(1+\frac{zf''(z)}{f'(z)}\right) \geq\frac{1-3r}{1-r^2}\geq\alpha,\]
 or \[\alpha r^2-3r+1-\alpha \geq0.\]
 The last inequality holds when  $r\leq R_{\C}$.

Now consider functions $f_0$ and $g_0$ defined by
\begin{equation}\label{8f0g0}
f'_0(z)=\frac{1+z}{(1-z)^2} \quad \text{and} \quad g_0(z)= \frac{z}{1-z}.
\end{equation}
Since $\left| f'_0(z)/g'_0(z)-1\right|=|z|<1$  and $g_0$ is convex,  the function
$f_0 \in \mathcal{F}_8$. Also
\[1+\frac{zf''_0(z)}{f'_0(z)}=\frac{1+3z }{1-z^2}.\]
  At  $z=-r=-R_{\C}$, then
\[\RE\left(1+\frac{zf''_0(z)}{f'_0(z)}\right)=\RE\frac{1+3z}{1-z^2}=\frac{1-3r}{1-r^2}=\alpha.\]
This shows that  the result in (a)  is sharp.

(b) In view of Lemma \ref{lm6}, the  circular disk \eqref{e18}  lies completely  inside the parabolic region $\{w:|w-1|<\RE w\}$ if
\[\frac{3r+r^2}{1-r^2}\leq\frac{1+r^2}{1-r^2}-\frac{1}{2}\]
or whenever
\begin{equation}\label{eq6}
r^2-6r+1\geq0.
\end{equation} The last inequality holds if $r\leq R_{\U}=R_{\mathcal{C}( 1/2)}=3-2\sqrt{2}$ for $z=-r=-R_{\U}$.
The function $f_0 $ given  by \eqref{8f0g0} satisfies,
\[
\left|\left(1+\frac{zf''_0(z)}{f'_0(z)}\right)-1\right| =\frac{3r-r^2}{1-r^2} =\frac{1-3r}{1-r^2}
 =\RE\left(1+\frac{zf''_0(z)}{f'_0(z)}\right),
\]
 and so the result in (b)  is sharp.
\end{proof}

\noindent{\bf Acknowledgment.} The work presented here was supported in parts by a Research University grant from Universiti Sains Malaysia, and by a research fellowship from the Council of Scientific and Industrial Research, New Delhi.


\begin{thebibliography}{99}

\bibitem{raviucv} R. M. Ali,  V. Ravichandran, Uniformly convex and uniformly starlike functions,    Math. Newsletter,   \textbf{21} (2011), no.~1, 16--30.

\bibitem{nav}R. M. Ali, N. Jain and V. Ravichandran, Radii of starlikeness associated with the lemniscate  of Bernoulli and the left-half plane, App. Math. Comput. {\bf  218} (2012), 6557--6565.

\bibitem{ali1} R. M. Ali, N. E. Cho, N. Jain and V. Ravichandran, Radii of starlikeness and convexity of functions defined by subordination with fixed second coefficients, Filomat,   \textbf{26} (2012), 553–-561.

\bibitem{bajpai} P. L. Bajpai\ and\ P. Singh, The radius of convexity of certain analytic functions in the unit disc, Indian J. Pure Appl. Math. {\bf 5} (1974), no.~8, 701--707

\bibitem{bhargava} G. P. Bhargava\ and\ S. L. Shukla, The radius of univalence of certain regular functions, Proc. Nat. Acad. Sci. India Sect. A {\bf 54} (1984), no.~3, 251--254.

\bibitem{causey} W. M. Causey\ and\ E. P. Merkes, Radii of starlikeness of certain classes of analytic functions, J. Math. Anal. Appl. {\bf 31} (1970), 579--586.

\bibitem{chen}M. P. Chen, The radius of starlikeness of certain analytic functions, Bull. Inst. Math. Acad. Sinica {\bf 1} (1973), no.~2, 181--190.

\bibitem{duren} P. L. Duren, {\it Univalent Functions}, Grundlehren der Mathematischen Wissenschaften, 259, Springer, New York, 1983.

\bibitem{goel} R. M. Goel, The radius of univalence of certain analytic functions, Tohoku Math. J. (2) {\bf 18} (1966), 398--403.

\bibitem{goelb} R. M. Goel, On the radius of univalence and starlikeness for certain analytic functions, J. Math. Sci. {\bf 1} (1966), 98--102.

\bibitem{goel72} R. M. Goel, Radius of univalence and starlikeness for certain analytic functions, Indian J. Math. {\bf 14} (1972), 15--19.

\bibitem{ucv}A. W. Goodman, On uniformly convex functions, Ann. Polon. Math. {\bf 56} (1991), no.~1, 87--92.

\bibitem{kaplan} W. Kaplan, Close-to-convex schlicht functions, Michigan Math. J. {\bf 1} (1952), 169--185 (1953).

\bibitem{krzyz} J. Krzy\.z\ and\ M. O. Reade, The radius of univalence of certain analytic functions, Michigan Math. J. {\bf 11} (1964), 157--159.


\bibitem{maminda-ucv}W. C. Ma\ and\ D. Minda, Uniformly convex functions, Ann. Polon. Math. {\bf 57} (1992), no.~2, 165--175.

\bibitem{mac2}T. H. MacGregor, A class of univalent functions, Proc. Amer. Math. Soc. {\bf 15} (1964), 311--317

\bibitem{mac1}  T. H. MacGregor, The radius of convexity for starlike functions of order ${1/2}$, Proc. Amer. Math. Soc. {\bf 14} (1963), 71--76.

\bibitem{mac4} T. H. MacGregor, The radius of univalence of certain analytic functions, Proc. Amer. Math. Soc. {\bf 14} (1963), 514--520.

\bibitem{mac3} T. H. MacGregor, The radius of univalence of certain analytic functions. II, Proc\ Amer. Math. Soc. {\bf 14} (1963), 521--524.

\bibitem{2}Z. Nehari, {\it Conformal Mapping}, McGraw-Hill, Inc., New York, Toronto, London, 1952.

\bibitem{owa2} S. Owa\ and\ H. M. Srivastava, Some generalized convolution properties associated with certain subclasses of analytic functions, J. Inequal. Pure Appl. Math. {\bf 3} (2002), no.~3, Article 42, 13 pp. (electronic).

\bibitem{ksp1} K. S. Padmanabhan, On the radius of univalence and starlikeness for certain analytic functions, J. Indian Math. Soc. (N.S.) {\bf 29} (1965), 71--80.

\bibitem{ksp2} K. S. Padmanabhan, On the radius of univalence and starlikeness for certain analytic functions. II, J. Indian Math. Soc. (N.S.) {\bf 29} (1965), 201--208.


 \bibitem{ratti68}J. S. Ratti, The radius of univalence of certain analytic functions, Math. Z. {\bf 107} (1968), 241--248. 

\bibitem{ratti70}J. S. Ratti, The radius of convexity of certain analytic functions, Indian J. Pure Appl. Math. {\bf 1} (1970), no.~1, 30--36.

\bibitem{ratti80}J. S. Ratti, The radius of convexity of certain analytic functions. II, Internat. J. Math. Math. Sci. {\bf 3} (1980), no.~3, 483--489.

\bibitem{vraviron}V. Ravichandran, F. R\o nning\ and\ T. N. Shanmugam, Radius of convexity and radius of starlikeness for some classes of analytic functions, Complex Variables Theory Appl. {\bf 33} (1997), no.~1-4, 265--280.


\bibitem{reade} M. O. Reade, On close-to-close univalent functions, Michigan Math. J. {\bf 3} (1955), 59--62.

\bibitem{reade65} M. O. Reade, S. Ogawa\ and\ K. Sakaguchi, The radius of convexity for a certain class of analytic functions, J. Nara Gakugei Univ. Natur. Sci. {\bf 13} (1965), 1--3.

\bibitem{ronn}F. R\o nning, Uniformly convex functions and a corresponding class of starlike functions, Proc. Amer. Math. Soc. {\bf 118} (1993), no.~1, 189--196.

\bibitem{ron5} F. R\o nning, A survey on uniformly convex and uniformly starlike functions, Ann. Univ. Mariae Curie-Sk\l odowska Sect. A {\bf 47} (1993), 123--134.

\bibitem{singh} R. Singh, Correction: ``On a class of starlike functions'', Compositio Math. {\bf 21} (1969), 230--231.

\bibitem{shah} G. M. Shah, On the univalence of some analytic functions, Pacific J. Math. {\bf 43} (1972), 239--250.

\bibitem{vravicmft}T. N. Shanmugam\ and\ V. Ravichandran, Certain properties of uniformly convex functions, in {\it Computational Methods and Function Theory 1994 (Penang)}, 319--324, Ser. Approx. Decompos., 5 World Sci. Publ., River Edge, NJ.

\bibitem{shukla} S. L. Shukla, On a class of certain analytic functions, Demonstratio Math. {\bf 17} (1984), no.~4, 887--896.

\bibitem{sokol96}J. Sok\'o\l\ and\ J. Stankiewicz, Radius of convexity of some subclasses of strongly starlike functions, Zeszyty Nauk. Politech. Rzeszowskiej Mat. No. 19 (1996), 101--105.

\bibitem{sokol09}J. Sok\'o\l, Coefficient estimates in a class of strongly starlike functions, Kyungpook Math. J. {\bf 49} (2009), no.~2, 349--353.

\bibitem{sokol09b} J. Sok\'o\l, Radius problems in the class ${\mathcal{S\!\!\!L}}\sp *$, Appl. Math. Comput. {\bf 214} (2009), no.~2, 569--573.

\bibitem{tuan} P. D. Tuan\ and\ V. V. Anh, Radii of starlikeness and convexity for certain classes of analytic functions, J. Math. Anal. Appl. {\bf 64} (1978), no.~1, 146--158.

\bibitem{ural} B. A. Uralegaddi, M. D. Ganigi\ and\ S. M. Sarangi, Univalent functions with positive coefficients, Tamkang J. Math. {\bf 25} (1994), no.~3, 225--230.

\end{thebibliography}
\end{document}